\documentclass{article}

\usepackage[centertags]{amsmath}
\usepackage{amsfonts}
\usepackage{amssymb}
\usepackage{amsthm}
\usepackage{enumitem}
\usepackage{newlfont}
\usepackage{mathtools}
\usepackage[top=1in, bottom=1in, left=1in, right=1in]{geometry}
\usepackage[colorlinks]{hyperref}
\usepackage{tikz}
\usetikzlibrary{arrows,shapes,decorations.pathreplacing}
\usepackage{pgf}
\usepackage{etex}

\tikzstyle{every picture}+=[remember picture]

\theoremstyle{plain}
\newtheorem{theorem}{Theorem}[section]

\newtheorem{corollary}[theorem]{Corollary}
\newtheorem{lemma}[theorem]{Lemma}

\theoremstyle{definition}
\newtheorem{definition}[theorem]{Definition}
\newtheorem{claim}[theorem]{Claim}

\newtheorem{remark}[theorem]{Remark}

\newcommand{\PK}[1][K]{\mathcal{P}_{#1}}

\newcommand\+{\mkern2mu}

\newcommand{\R}{\mathbb{R}}

\newcommand{\Z}{\mathbb{Z}}
\newcommand{\C}{\mathbb{C}}
\newcommand{\Q}{\mathbb{Q}}

\newcommand{\Hidden}[1]{}
\newcommand{\F}{\mathcal{F}}
\newcommand{\Sl}{\mathcal{S}}

\newcommand{\ep}{\varepsilon}
\newcommand{\1}{\mathbf 1}

\newcommand{\tthe}[1][j]{\tilde{\theta_{#1}}}
\newcommand{\trho}{\tilde{\rho}}
\newcommand{\tr}{\mathrm{T}}

\DeclareMathOperator{\Tr}{Tr}
\DeclareMathOperator{\im}{Im}
\DeclareMathOperator{\Id}{Id}
\DeclareMathOperator{\cl}{cl}

\DeclarePairedDelimiter{\abs}{\lvert}{\rvert}
\DeclarePairedDelimiter{\znorm}{\lVert}{\rVert}
\DeclarePairedDelimiter{\sprod}{\langle}{\rangle}

\begin{document}

\title{Pretty good fractional revival via magnetic fields: \\ theory and examples}
\author{Whitney Drazen, Mark Kempton\footnote{Department of Mathematics, Brigham Young University, Provo, UT, mkempton@mathematics.byu.edu}, Gabor Lippner\footnote{Department of Mathematics, Northeastern University, Boston, MA, g.lippner@northeastern.edu}}
\date{}

\maketitle

\begin{abstract}
We develop the theory of pretty good quantum fractional revival in arbitrary sized subsets of a graph, including the theory for fractional cospectrality of subsets of arbitrary size.  We use this theory to give conditions under which a magnetic field can induce pretty good fractional revival, and give several examples.
\end{abstract}

\section{Introduction}
``Fractional revival is a quantum transport phenomenon important for entanglement generation in spin networks''~\cite{ChanCoutinhoTamonVinetZhan}. In a continuous time quantum walk fractional revival (FR) refers to the situation when the walk ``preserves'' a subset at a certain moment in time. That is, there is a subset $K$ of the nodes and a time $t$ such that if the initial state of the walk is supported on $K$ then it is also supported on $K$ at time $t$. For entanglement generation one would, typically, be interested in starting the walk from a single vertex $v\in K$ and obtaining a superposition of the nodes in $K$ at time $t$. For more background and a comprehensive characterization of FR see~\cite{chan2020fundamentals}.

Fractional revival is, in a sense, a relaxation of perfect state transfer (PST). Nevertheless, finding examples of FR turned out to be nearly as difficult as for PST. This naturally led to the study of further relaxations of the phenomenon. Just as with PST, one can introduce an asymptotic variant that has now been routinely dubbed ``pretty good \underline{\phantom{PST}}'' in the literature \cite{Vinet2012,Godsil2012}. Informally, \emph{pretty good fractional revival} (PGFR) requires a sequence of times at which the walk is closer and closer to actually preserving the subset $K$. A complete characterization of PGFR between a pair of nodes on paths and cycles was given in~\cite{chan2020approximate}. 

In a series of papers~\cite{us, invol, eisenberg2019pretty} (some of) the present authors have developed methods to construct examples of pretty good (or even perfect) state transfer using a diagonal perturbation of the matrix - sometimes referred to as a magnetic field in the context of quantum spin networks. The goal of the current paper is to extend these ideas to the case of PGFR. In particular, we further develop the theory of PGFR to obtain a practical, verifiable condition that guarantees that a subset of nodes will exhibit pretty good fractional revival after adding a ``generic'' constant diagonal perturbation to the matrix. 

In order to achieve this goal, we
\begin{itemize}
\item devise a way to split the characterization of PGFR into somewhat separate ``eigenvector'' and ``eigenvalue'' conditions,
\item introduce the notion of fractional cospectrality and provide a comprehensive characterization of it in order to construct families of graphs that satisfy the eigenvector part of the condition,
\item prove a suitable generalization of the well-known Kronecker condition for pretty good state transfer to the case of pretty good fractional revival,
\item extend the field-trace method of~\cite{invol} into a tool that allows one to verify this new Kronecker-type condition if certain factors of the characteristic polynomial are irreducible
\item prove that under a generic diagonal perturbation the relevant factors are indeed irreducible. 
\end{itemize}

The paper is structured as follows: in Section~\ref{sec:pgfr} we introduce pretty good fractional revival (PGFR) and provide a spectral characterization. Then we generalize Kronecker's criterion to our setting and, using the field-trace method developed in~\cite{invol}, we derive a sufficient condition for PGFR based on the irreducibility of certain factors of the characteristic polynomial together with a trace and degree condition on these factors.

In Section~\ref{sec:gen-cospectral} we explain our generalization of the idea of cospectrality to the fractional setting, and prove a theory analogous to the characterizations of the original notion. This allows us to prove that for fractionally cospectral subsets, under suitable diagonal perturbations of the adjacency matrix, the factors of the characteristic polynomial relevant to PGFR are indeed irreducible. 

Finally, in Section~\ref{sec:examples} we construct examples where we can prove fractional cospectrality of certain subsets and also verify the trace and degree condition of Theorem~\ref{thm:tr/deg}, thereby guaranteeing PGFR in these graphs.

\section{Pretty good fractional revival}\label{sec:pgfr}


We work in the following general setting: fix an index set $X$ and consider a real symmetric matrix $M \in \R^{X\times X}$. 

\begin{definition}\label{def:pgfr}
Let $K \subset X$ be a subset of indices. We say that the $X \times X$ matrix $M$ exhibits pretty good fractional revival with respect to $K$ if 
\[ \cl \{ \exp(i t M)_{K \times K} : t \geq 0\} \cap U(K) \nsubseteq \{ \rho \Id_{K \times K} : \rho \in \C\}, \]
that is, if we take the family of  $K\times K$ submatrices of $\exp(i t M)$ for all $t \geq 0$, then the closure of this family contains a unitary matrix with at least two distinct eigenvalues.

Equivalently, there is a $K \times K$ unitary matrix $H$ with at least two distinct eigenvalues, and a sequence $0 < t_1 \leq t_2 \leq  \dots$ such that $\lim_{k \to \infty} \exp(i t_k M)_{K \times K} = H$. This includes fractional revival if $t = t_1 = t_2 = \dots$. The convergence can be understood entry wise or, equivalently, with respect to any standard matrix norm.

Since the $K \times K$ submatrix of a unitary matrix $A$ is unitary if and only if $A$ is block-diagonal relative to $K$, a further equivalent way to describe pretty good fractional revival is to require that there is a sequence $0 < t_1 \leq t_2 \leq \dots$ and a matrix $A$ that is block-diagonal relative to $K$ such that $\lim_{k\to \infty} \exp(it_k M) = A$, and $A_{K\times K}$ has at least two distinct eigenvalues.

\end{definition}

\begin{remark}
This generalizes the concept of pretty good fractional revival on 2 nodes of a graph from \cite{chan2020approximate}:
Two vertices $u$ and $v$ of a graph $G$ with adjacency matrix $A$ exhibit pretty good fractional revival if, for all $\epsilon>0$ there is some time $t>0$ such that
\[
|e^{itA}(u,u)|^2+|e^{itA}(u,v)|^2 > 1-\epsilon.
\]
\end{remark}

\begin{remark}
It would be perhaps more natural to require a non-diagonal $K \times K$ unitary matrix in the closure of $\exp(it M)_{K \times K}$ instead of simply one that's not the multiple of the identity. However, it turns out that for primitive matrices (eg for adjacency matrices of connected graphs) this stronger requirement is equivalent to the one in Definition~\ref{def:pgfr}. See the second part of Theorem~\ref{thm:pgfr_pk} for an explanation of this.
\end{remark}

Notice that $\{ \exp(i t M) : t \geq 0\}  \subset \sprod{M}$ is a bounded subset of the (finite dimensional) polynomial algebra generated by $M$. Hence if $M$ exhibits pretty good fractional revival with respect to $K$ and $\lim_{k \to \infty} \exp(i t_k M)_{K \times K} = H$ then there is a unitary matrix $\hat{H} \in \sprod{M}$ that is block diagonal relative to $K$, and $\Id_{K \times K} \neq H = \hat{H}_{K\times K}$. 

\subsection{Non-degenerate partition of the spectrum}

Here we provide a spectral characterization of pretty good fractional revival that can be summarized as follows: the subset $K$ induces a natural partition $\PK$ of the eigenvalues of $M$, and pretty good fractional revival is exhibited with respect to $K$ if and only if a certain simultaneous approximation problem is solvable in this partition. 

Let us start by recalling some important notions and facts from~\cite{chan2020fundamentals}.

\begin{definition}\label{def:partition}
Let $M = \sum_{i=1}^d \theta_i E_i$ be the spectral decomposition of $M$, and $K \subset X$ a non-empty set of indices.
\begin{enumerate}
\item We denote by $D_K$ the $X \times X$ diagonal matrix whose entries on the diagonal are 1 in $K$ and 0 outside of $K$. 
\item The \emph{eigenvalue support} of $K$ is the binary relation \[ \Phi_K = \{ (\theta_r, \theta_s) : E_r D_K E_s \neq 0 \}.\]
\item Define $\PK = (\Pi_0, \Pi_1,\dots, \Pi_s)$ to be the partition of $\{1,2,\dots, d\}$ where 
\[ \Pi_0 = \{ i: (E_i)_{K\times K} = 0 \},\] and $\Pi_1,\dots, \Pi_s$ are the remaining equivalence classes of the transitive closure of $\Phi_K$.
\end{enumerate}
\end{definition}

\begin{lemma}[Lemma 2.5 and Theorem 2.10 from~\cite{chan2020fundamentals}]\label{lem:decomposable}
$A = \sum_j c_j E_j$ is block-diagonal relative to $K$ if and only if the $c_j$s are equal to each other within each part $\Pi_r : 1\leq r \leq s$.
\end{lemma}

\begin{definition}\label{def:non-deg}
The partition $\PK$ is \emph{non-degenerate} if there is a \emph{mod 1 non-constant} vector $(\rho_1, \dots, \rho_s) \in \R^s$ such that for all $\ep > 0$ there is a $t > 0$ so that for any $1 \leq r \leq s$ 
\begin{equation}\label{eq:approximate} \forall j \in \Pi_r : \znorm{ t \cdot \theta_j - \rho_r} < \ep.\end{equation}
where $\znorm{x} = \min\{\abs{x-n} : n \in \Z\}$ is the distance of $x$ to the nearest integer. In particular, $s \geq 2$ is required.

\end{definition}

\begin{theorem}\label{thm:pgfr_pk}
The matrix $M$ exhibits pretty good fractional revival with respect to $K$ if and only if the partition $\PK$ is non-degenerate. Furthermore, if $M$ is primitive then pretty good fractional revival also implies that $\cl \{ \exp(i t M)_{K \times K} : t \geq 0\} \cap U(K)$ contains a non-diagonal matrix.
\end{theorem}

\begin{proof}
First assume $M$ exhibits pretty good fractional revival with respect to $K$, that is, there is a block-diagonal matrix $A$ such that $H  = A_{K\times K}$ has at least two distinct eigenvalues, and a sequence $0 < t_1 \leq t_2 \leq \dots$ such that $\lim_{k\to \infty} \exp(2\pi i t_k M) = A$. Since $\exp(2\pi i t M) = \sum_j \exp(2\pi i t \theta_j) E_j$, it follows that $\mu_j = \lim_{k\to \infty} exp(2\pi i t_k \theta_j)$ exists for all $j$, and that $A = \sum \mu_j E_j$. By Lemma~\ref{lem:decomposable} this implies that there exist reals $\rho_1, \dots, \rho_s$ such that $\mu_j = \exp(2\pi i \rho_r)$ for all $j \in \Pi_r$. Thus $\znorm{t_k \theta_j - \rho_r} \to 0$ as $k\to \infty$ for all $j \in \Pi_r$. Finally, since $H$ has at least two distinct eigenvalues, all the $\mu_j$s can't be identical, which implies that all the $\rho_r$ cannot be congruent mod 1. This proves the only if part.

Conversely, if we are given $(\rho_1,\dots, \rho_s)$ that proves the non-degenerateness of $\PK$, let $t_k$ be the time for which \eqref{eq:approximate} holds for $\ep = 1/k$. Then, clearly, $\lim_{k\to \infty} \exp(2\pi i t_k \theta_j) = \exp(2\pi i \rho_r)$ for all $j \in \Pi_r$. Choose a subsequence for which $\gamma_j = \lim_{k \to \infty} \exp(2\pi i t_k \theta_j)$ also exists for all $j \in \Pi_0$. (This can be done simply by compactness.) Then 
\begin{equation}\label{eq:closure} 
A =  \lim_{k\to \infty} \exp(2\pi i t_k M) = \sum_{j \in \Pi_0} \gamma_j E_j + \sum_{r=1}^s \left(\exp(2\pi i \rho_r) \sum_{j \in \Pi_r} E_j\right)
\end{equation} 
is block-diagonal by Lemma~\ref{lem:decomposable}, and since not all the $\rho_r$s are congruent mod 1, the restriction $H = A_{K\times K}$ is not a scalar multiple of the identity.

To prove that the non-degeneracy of $\PK$ implies the existence of non-diagonal matrices in $\cl \{ \exp(i t M)_{K \times K} : t \geq 0\} \cap U(K)$, it suffices to show that $H = A_{K\times K}$ is non-diagonal. The argument below is adapted from the proof of Lemma 2.9 in~\cite{chan2020fundamentals}. Let $x \in X$ and let $e_x$ denote the corresponding standard basis vector. Then, for any $j \in \Pi_r$ we can write
\[ H(x,x) E_j e_x = E_j A e_v = \exp(2\pi i \rho_r) E_j e_x,\] which implies that $E_j e_x = 0$ unless $\exp(2\pi i \rho_r) = H(x,x)$. By symmetry the same holds for $e_x^{\tr} E_j$.  

Now suppose, for a contradiction, that $H$ is diagonal. Since it is not a scalar multiple of the identity according to the first half of the theorem, we can find two elements $x,y \in X$ such that $H(x,x) \neq H(y,y)$. Thus $e_y^{\tr} E_j e_x = 0$ for all $j$, since $\exp(2\pi i \rho_r)$ can't equal to both $H(x,x)$ and $H(y,y)$. Then, however, $e_y^{\tr} M^n e_x = \sum_j \theta_j^s e_y^{\tr} E_j e_x = 0$ for any integer $n\geq 0$ which contradicts the primitivity of $M$. 
\end{proof}

\subsection{A number theoretic characterization}

In this section we develop a method to verify that a partition is degenerate. This can be considered as the generalization of previous results for the case of pretty good state transfer. Its basis, as in earlier results, is the following number theoretic lemma due to Kronecker.

\begin{lemma}[Kronecker]\label{lem:Kron}
Let $\theta_1,...,\theta_k$ and $\zeta_1,...,\zeta_k$ be arbitrary real numbers.  For an arbitrarily small $\ep > 0$, the system of inequalities
\[
\znorm{\theta_ j y - \zeta_j} < \ep \ (j=1,...,k),
\]
has a solution $y$ if and only if, for integers $\ell_1,...,\ell_k$, 
\[
\ell_1\theta_1+\cdots+\ell_k\theta_k = 0,
\]
implies
\[
\znorm{\ell_1\zeta_1 + \cdots + \ell_k\zeta_k}= 0.
\]
\end{lemma}

\begin{lemma}\label{thm:kronecker}
Given a sequence of real numbers $\theta_1, \dots, \theta_k$ and a partition $\mathcal{P} = (\Pi_1,\dots,\Pi_s)$ of $\{1,2,\dots, k\}$, the following are equivalent:
\begin{enumerate}[label=\roman*)]
\item $\mathcal{P}$ is non-degenerate in the sense of Definition~\ref{def:non-deg}.
\item There is a pair of indices $1 \leq r_1, r_2 \leq s$ such that no sequence integers $\ell_1, \ell_2, \dots, \ell_k$ satisfies all of the following:
\begin{enumerate}[label=\alph*)]
\item $\sum_1^k \ell_j \theta_j = 0$,
\item $\sum_{j\in\Pi_r}\ell_j=0$ for all $r \neq r_1,r_2$ 
\item $\sum_{j\in\Pi_{r_1}}\ell_j=-1$ and  $\sum_{j\in\Pi_{r_2}}\ell_j=1$.
\end{enumerate}

\end{enumerate}
\end{lemma}

\begin{remark}
Note that this lemma is a direct generalization of \cite{chan2020approximate}[Theorem 2.4] which only considers the $s=2$ case. The only essential new ingredient in the following proof is the use of integer lattices in place of subgroups of $\Z$.
\end{remark}

\begin{proof}
First, suppose that $\mathcal{P}$ is non-degenerate as witnessed by the sequence $(\rho_1,\dots, \rho_s)$ where $\znorm{\rho_{r_1} - \rho_{r_2}} > 0$ for some $r_1, r_2$. Assume, for the same $r_1,r_2$, that there exists integers $\ell_1,\dots, \ell_k$ satisfying the above criteria. Let $\ep = \ep_0/(k \max\{|\ell_j|\})$, and choose $t$ such that  for any $1 \leq r \leq s$ 
\begin{equation*} \forall j \in \Pi_r : \znorm{ t \cdot \theta_j - \rho_r} < \ep.
\end{equation*}
Then for all $j\in \Pi_r$
\[ \znorm{ t \ell_j \theta_j - \ell_j \rho_r} < \ep_0/k\] 
and thus
\[ \znorm*{ \sum_1^k t \ell_j \theta_j - \sum_{r=1}^s \rho_r \sum_{j \in \Pi_r} \ell_j } < \ep_0.
\]
Here the left and the right sums are both $\rho_{r_1}-\rho_{r_2}$ and the middle is 0. Thus $\znorm{\rho_{r_1} - \rho_{r_2}} < \ep_0$. Since this holds for any $\ep_0$, we get that $\znorm{\rho_{r_1} - \rho_{r_2}}=0$, contradicting the choice of $r_1,r_2$. This proves the $i) \Rightarrow ii)$ implication.

Next, we prove $ii) \Rightarrow i)$. Assume, without loss of generality, that $r_1 = 1$ and $r_2 = 2$ is a pair of indices for which $ii)$ holds. Further assume, again without loss of generality, that $\theta_1 \in \Pi_1$.  Let $\tthe = \theta_j - \theta_1 : j = 2, \dots, k$. Note that 
\begin{equation}\label{eq:ell1} \sum_{j=2}^k \ell_j \tthe = 0 \mbox{ implies } \sum_1^k \ell_j \theta_j = 0 \mbox{ and } \sum_1^k \ell_j = 0 \mbox{ where $\ell_1 = -\sum_2^k \ell_j$}.\end{equation}

Consider the following integer lattice 
\[ \Sl = \{ (a_2, \dots, a_s) \in \Z^{s-1} | \exists\+ \ell_2, \dots, \ell_k \in \Z : \sum_2^k \ell_j \tthe = 0 \mbox{ and } a_r = \sum_{j \in \Pi_r} \ell_j  \mbox{ for all $2 \leq r \leq s$} \}.\]
By \eqref{eq:ell1} we see that $(1,0,0,\dots,0) \not \in \Sl$ hence $\Sl \subsetneq \Z^{s-1}$. This implies that the dual lattice $\Sl^* \supsetneq \Z^{s-1}$. (For instance, because the determinant of $\Sl$ has to be greater than 1, and thus the determinant of the dual has to be less than 1 in absolute value so it can't be an integer lattice.) In other words, there exists vector $(\trho_2,\dots, \trho_s)$ that is not congruent to $(0,0,\dots,0)$ mod 1, such that $\sum_2^s \trho_r a_r \in \Z$ for all $(a_2,\dots,a_s) \in \Sl$. 

Now let 
\[ \zeta_j = \left\{ 
\begin{array}{lll}
\trho_r & \mbox{ if } & j \in \Pi_r, r\geq 2\\
0 & \mbox{ if } & j \in \Pi_1
\end{array}
\right.
\]
If $\ell_2, \dots, \ell_k$ are integers such that $\sum_2^k \ell_j \tthe = 0$, then $(a_2,\dots,a_s) \in \Sl$ for $a_r = \sum_{j \in \Pi_r} \ell_r : r=2,\dots, s$. Hence  
\[ \sum_2^k \ell_j \zeta_j = \sum_{r=1}^s \left( \trho_r \sum_{j \in \Pi_r} \ell_j \right) = 0 + \sum_{r=2}^s \trho_r a_r  \in \Z\] by the choice of $(\trho_2,\dots, \trho_s) \in \Sl*$. So Lemma~\ref{lem:Kron} implies that for any $\ep > 0$ there is a $t = t(\ep)$ such that we have $\abs{t \tthe - \trho_r} < \ep \pmod{1} $ for all $j = 2,\dots,k$ where $\Pi_r$ is the partition containing the index $j$. To finish the proof, let $\rho_1$ be a mod 1 accumulation point of the sequence $t(\ep) \theta_1$ as $\ep \to 0$. Then for any $\ep >0 $ there is a $t = t(\ep)$ such that $\znorm{t \theta_1 -\rho_1} < \ep$ and $\abs{t \tthe - \trho_r} < \ep \pmod{1} $ for all $j$, simultaneously. Set $\rho_r = \trho_r + \rho_1$. Then $(\rho_1, \dots, \rho_s)$ is not the constant vector mod 1. Since $\znorm{t(\ep) \theta_j - \rho_r } \leq \znorm{t(\ep) \tthe - \trho_r} + \znorm{t(\ep) \theta_1 - \rho_1} < 2\ep$ for $j = 2,\dots,k$ where $\Pi_r$ is the partition containing the index $j$. Thus the partition is non-degenerate.
\end{proof}

In general, verifying ii) is difficult. Here we generalize a tool from~\cite{eisenberg2019pretty} that allows to show ii) holds under certain conditions that are easy to verify.

\begin{theorem}\label{thm:tr/deg}

Fix a field $\F$, a sequence of real numbers $\theta_1, \dots, \theta_k$ and a partition $\mathcal{P} = (\Pi_1,\dots,\Pi_s)$ of $\{1,2,\dots, k\}$. Suppose there are polynomials $P_1,P_2,\dots, P_s  \in \F[x]$ that are irreducible over $\F$ and such that $\Pi_r$ is exactly the set of roots of $P_r$ for each $1\leq r \leq s$. 

If for some pair of integers $1 \leq r_1, r_2 \leq s$
\[
\frac{\Tr(P_{r_1})}{\deg(P_{r_1})}\neq\frac{\Tr(P_{r_2})}{\deg(P_{r_2})},
\]
where $\Tr$ denotes the trace (i.e. the sum of roots) of a polynomial, then the partition $\mathcal{P}$ is non-degenerate. 
\end{theorem}

\begin{proof}
We verify ii) of Lemma~\ref{thm:kronecker} for $r_1,r_2$ via the field trace a method introduced in~\cite{invol}. For a field extension $\mathcal{K}$ of $\mathcal{F}$, the \emph{field trace} is a linear functional $\Tr_{\mathcal K/\mathcal F}: \mathcal K \rightarrow \mathcal F$  defined for each element  $\alpha \in \mathcal{K}$ as the trace of the $\mathcal{F}$-linear map $x\mapsto \alpha x$. See \cite{lang_book} for details about the field trace.

For each $1\leq r \leq s$ let  $\mathcal{L}_r$ denote the splitting field of $P_r$ over $\F$. Since $P_r$ is irreducible, for any $j \in \Pi_r$ we have 
\[ \Tr_{\mathcal{L}_r / \F}(\theta_j) = \frac{[\mathcal L_r:\mathcal{F}]}{\deg P_j}\sum_{j \in \Pi_r} \theta_j =  \frac{[\mathcal L_r:\mathcal{F}]}{\deg P_j} \Tr(P_r)\]
according to Lemma A.1 of \cite{eisenberg2019pretty}.
Let us suppose for a contradiction that there are integers $\ell_j \in \Z : j=1,\dots, k$ satisfying 
\begin{align}
\sum_{j=1}^k\ell_j \theta_j &= 0 \label{eq:lincomb}\\*
\sum_{j \in \Pi_{r_1}} \ell_j &= 1 \text{ and } \sum_{j \in \Pi_{r_2}} \ell_j = -1\label{eq:pm1}\\
\sum_{j \in \Pi_r} \ell_j &=0 \text{ for all } r \neq r_1, r_2. \label{eq:sum0}
\end{align}

Let $\mathcal K/\mathcal F$ the smallest field extension containing $\mathcal L_1,...,\mathcal L_s$. We apply $\Tr_{\mathcal{K}/\F}$ to \eqref{eq:lincomb}. Then, according to \eqref{eq:pm1},  \eqref{eq:sum0}, and the basic properties of the field trace from Lemma A.1 in~\cite{eisenberg2019pretty},
\begin{align*}
0&=\Tr_{\mathcal K/\mathcal F'}\left( \sum_{j=1}^k \ell_j \theta_j \right) = \sum_{r=1}^s[\mathcal K:\mathcal L_r]\Tr_{\mathcal L_r/F}\left(\sum_{j \in \Pi_r} \ell_j \theta_j \right)\\
&=\sum_{r=1}^s\left([\mathcal K:\mathcal L_r] \sum_{j \in \Pi_r} \ell_j \Tr_{\mathcal L_r/F}\left(\theta_j \right)\right)\\ 
&=\sum_{r=1}^s\left(\frac{[\mathcal K:\mathcal L_r][\mathcal L_r:\mathcal F]}{\deg(P_r)}\Tr(P_r)\sum_{j \in \Pi_r} \ell_j \right)\\
&=[\mathcal K:\mathcal F]\sum_{r=1}^s\frac{\Tr(P_r)}{deg(P_r)}\sum_{j \in \Pi_r} \ell_j 
=[\mathcal K:\mathcal F']\left(\frac{\Tr(P_1)}{deg(P_1)}-\frac{\Tr(P_2)}{deg(P_2)}\right).
\end{align*}
This contradicts $\frac{\Tr(P_1)}{deg(P_1)} \neq \frac{\Tr(P_2)}{deg(P_2)}$ hence ii) of Lemma~\ref{thm:kronecker} holds for $r_1,r_2$, and thus $\mathcal{P}$ is non-degenerate.

\end{proof}

\section{Generalized cospectrality}\label{sec:gen-cospectral}

It is now a well-known fact that perfect state transfer can be characterized by an eigenvector and an eigenvalue condition. The eigenvector condition is called strong cospectrality (see~\cite{godsil_smith_2017}), a strengthening of the classical notion of cospectrality of two nodes. In~\cite{invol, eisenberg2019pretty} cospectrality was used as a starting point to construct examples for pretty good state transfer. 

In~\cite{chan2020fundamentals} the study of fractional revival between two nodes led to a generalization of both cospectrality and strong cospectrality of two nodes to the fractional setting. Further, decomposability was identified as the correct generalization of strong (fractional) cospectrality to arbitrary subsets. However, the analogous extension of the theory of cospectrality was not discussed.

\subsection{$H$-cospectrality}\label{sec:h-cospectral}

Since our goal is to construct examples that admit pretty good fractional revival, in this section we complete the picture by developing the theory of (fractional) cospectrality for arbitrary subsets. It turns out that most features of the classical theory carry over to the general setting. 

Let $K \subset X$ be a subset of indices. For the sake of the applications later on, it turns out to be simpler to work in the more general setting of complex vector spaces and normal matrices. We consider $\C^X$ and $\C^K$ equipped the usual Hermitian scalar product $\sprod{v,w} := v^* w \in \C$. For a vector $v \in \C^X$ (respectively a matrix $A \in \C^{X \times X}$) we denote $\tilde{v} = v_K$ (respectively $\tilde{A} = A_{K\times K}$) its restriction to the subset $K$. Conversely, given a vector $v \in \C^K$ we denote $\hat{v} \in \C^X$ its extension by 0s to the other coordinates of $X$. Note that 
\begin{equation}\label{eq:tilde} \widetilde{A \hat{v}} = \tilde{A} v \mbox{ and } \sprod{A \hat{v}, \hat{w}} = \sprod{\tilde{A} v, w} \end{equation} for any $A \in \C^{X \times X}$ and $v,w \in \C^K$.

Let $M \in \C^{X \times X}$ and $H \in \C^{K\times K}$ be normal matrices with spectral decompositions $M = \sum_{i=1}^d \theta_i E_i$ and $H = \sum_{j=1}^r \rho_j F_j$. The $E_i$ and $F_j$ are self-adjoint projections. 

\begin{definition}\label{def:cospectral}
We say that $K$ is $H$-cospectral in $M$ (or $H$-cospectral, for short) if there is an orthonormal (with respect to the Hermitian scalar product) eigenbasis $\psi_1, \dots, \psi_{|X|}$ such that $\tilde{\psi_j}$ is either 0 or an eigenvector of $H$ for all $j = 1, \dots, \abs{X}$.\end{definition}

\begin{remark}~
\begin{itemize}
\item Clearly, the dependence on $H$ is only via its spectral idempotents $\{F_j\}$. However, it is often more convenient to refer to the matrix $H$ instead of a collection of projectors.
\item This generalizes fractional cospectrality between two nodes (the $\abs{K} =2$ case) introduced in~\cite{chan2020fundamentals} to subsets of arbitrary size. In particular, if \begin{equation}\label{eq:0110} H = \left( \begin{array}{cc} 0 & 1\\ 1 & 0 \end{array} \right)\end{equation} then we recover the classical notion of cospectrality.
\end{itemize}
\end{remark}

\begin{theorem}\label{thm:frac_cosp}
Let $K \subset X, M, H$ as above. The following are equivalent:
\begin{enumerate}
\item\label{it:frco} $K$ is fractionally cospectral with respect to $H$.
\item\label{it:HMk} $H \widetilde{M^k} = \widetilde{M^k} H$ for all $k$.
\item\label{it:HEi} $H \tilde{E_i} = \tilde{E_i} H$ for all $j$.
\item\label{it:FjEi} $F_j \tilde{E_i} = \tilde{E_i}F_j$ for all $i,j$.
\item\label{it:Eivw} For any $v,w$ eigenvectors of $H$ belonging to different eigenvalues, $\sprod{\tilde{E_i} v, w} = 0$ for all $i$.
\item\label{it:basisrk} For each $i$, there is an orthonormal basis of $\im E_i$ that contains exactly $\dim \{ E_i \hat{v} : v \in
\im F_j \}$ vectors that satisfy $0\neq \tilde{\psi} \in \im F_j $ for each $j$, and the rest of the basis elements satisfy $\tilde{\psi} = 0$.
\item\label{it:Mkw} For any $v,w$ eigenvectors of $H$ belonging to different eigenvalues, the subspaces $\sprod{M^k \hat{v} : k=0,1,\dots}$ and $\sprod{M^k \hat{w} : k=0,1,\dots}$ are orthogonal.
\end{enumerate}
\end{theorem}

\begin{proof} We prove implications in a cyclic order.
\paragraph{$\ref{it:frco} \implies \ref{it:HMk}$:} Let $\psi_1, \dots, \psi_{\abs{X}}$ as in Definition~\ref{def:cospectral} with corresponding eigenvalues $\lambda_1, \dots , \lambda_{\abs{X}}$.
Then $M^k = \sum \lambda_i^k \psi_i \psi_i^*$ and hence $\widetilde{M^k} = \sum \lambda_i^k \tilde{\psi_i} \tilde{\psi_i}^*$. The $\tilde{\psi_i}$s are all eigenvectors of $H$. If $H v = \rho v$ then $H v v^* = \rho v v^* = v (\bar{\rho v})^* = v (H^* v)^*  = v v^* H$ by the normality of $H$. Thus $H$ commutes with all $\tilde{\psi_i} \tilde{\psi_i}^*$ terms, and in turn also with $\widetilde{M^k}$ for any $k$.

\paragraph{$\ref{it:HMk} \implies \ref{it:HEi}$:} This follows since each $E_i$ is a polynomial of $M$, and thus $\tilde{E_i}$ is a linear combination of the $\widetilde{M^k}$s.

\paragraph{$\ref{it:HEi} \implies \ref{it:FjEi}$:} This follows since each $F_j$ is a polynomial of $H$.

\paragraph{$\ref{it:FjEi} \implies \ref{it:Eivw}$:} Let $v$ be an eigenvector in the $F_j$ eigenspace of $H$. That is, $F_j v = v$. Then $F_j w= 0$ since $w$ is in a different eigenspace. Now, using the self-adjointness of the $F_j$, we get $\sprod{\tilde{E_i} v, w} = \sprod{\tilde{E_i} F_j v, w} = \sprod{F_j \tilde{E_i} v, w} = \sprod{\tilde{E_i} v, F_j w} = 0$ as claimed.

\paragraph{$\ref{it:Eivw} \implies \ref{it:basisrk}$:} Let $E = E_i$ for some fixed $i$. For any $j$ consider the subspace $\mathcal{S}_j \subset \im E$ defined as
\[ \mathcal{S}_j = \{ E \hat{v} : v \in \im F_j \}. \]
If $v \in \im F_{j_1}$ and $w \in \im F_{j_2}$ for some $j_1 \neq j_2$, then $\sprod{E \hat{v}, E\hat{w}} = \sprod{E \hat{v}, \hat{w}} = \sprod{\tilde{E} v, w} = 0$ by the condition. Hence $\mathcal{S}_{j_1} \bot \mathcal{S}_{j_2}$. Thus we can pick an orthonormal basis in each of the $\mathcal{S}_j$s as well as the orthogonal complement of $\oplus \mathcal{S}_j$ in $\im E$. We claim that the union of these bases satisfies our requirements. To show this, consider $u \in \im E$ that is orthogonal to $\mathcal{S}_j$, and take any $v \in \im F_j$. Then
\[ 0 = \sprod{u, E \hat{v}} = \sprod{E u, \hat{v}} = \sprod{u, \hat{v}} = \sprod{\tilde{u},v}\] and thus $\tilde{u}$ is orthogonal to $\im F_j$. This means that if $u \in \mathcal{S}_j$ then it is orthogonal to $\mathcal{S}_{l}$ for all $l \neq j$, hence $\tilde{u}$ is orthogonal to $\im F_l$ for all $l \neq j$. That is only possible if $\tilde{u} \in \im F_j$. If $u = E\hat{v} \neq 0$, then $0 < \sprod{E \hat{v}, E \hat{v}} = \sprod{\hat{v},E\hat{v}} = \sprod{v, \tilde{u}}$, thus $\tilde{u} \neq 0$ as claimed. Lastly if $u \in \im E$ is orthogonal to all $\mathcal{S}_j$s, then by the same argument $\tilde{u}$ is orthogonal to all $\im F_j$s, hence $\tilde{u} = 0$ in this case.

\paragraph{$\ref{it:basisrk} \implies \ref{it:frco}$:} The union for all $i$ of the bases defined in \ref{it:basisrk} obviously satisfies Definition~\ref{def:cospectral}.

\paragraph{$\ref{it:Eivw} \Longleftrightarrow \ref{it:Mkw}$:} This follows since $\tilde{E_i}$ is a linear combination of the $\widetilde{M^k}$s as we have seen before, and vice versa: $\widetilde{M^k}$ is obviously a linear combination of the $\tilde{E_i}$s. Finally $\sprod{M^a \hat{v}, M^b \hat{w}} = \sprod{M^{a+b} \hat{v}, \hat{w}} = \sprod{\widetilde{M^{a+b}}v, w}$, so the latter is 0 for all $a,b$ if and only if the two subspaces are orthogonal. 
\end{proof}

The following observation will be useful later, but we state it here since the computation is essentially the same as in the proof of $\ref{it:Eivw} \implies \ref{it:basisrk}$ above.

\begin{claim}\label{claim:nonzeroEv}
Let $v \in \im F_j$. Then 
\[ \sprod{E_i \hat{v}, E_i \hat{v}} = \sprod{\hat{v}, E_i \hat{v}} = \sprod{v, \widetilde{E_i \hat{v}}} = \sprod{ v, \tilde{E_i} v},\]
hence  $E_i \hat{v} \neq 0$ if and only if its restriction to $K$ is non-zero. And in this case the restriction, $\tilde{E_i} v$, contains a component parallel to $v$.
\end{claim}

\begin{remark}\label{rem:equiv_def}
In the case when $|K| = 2$ and $H$ is a 2-by-2 matrix with eigenvectors $(p,q)$ and $(-q, p)$, our definition of $K$-cospectral coincides with what is called a fractionally cospectral pair of nodes in~\cite[Theorem 3.3, condition (ii)]{chan2020fundamentals}, hence this is indeed a direct generalization of that notion. 
\end{remark}

\subsection{A factorization of $\phi(M,t)$}

\begin{definition}
Let $K$ be $H$-cospectral in $M$. Define $b_{i,j} = \dim \{ E_i \hat{v} : v \in
\im F_j \}$ for any $1 \leq j \leq r$, and $b_{i,0} = \dim \{u \in \im E_i : \tilde{u} = 0\}$.  Let $P_j(t) = \prod_{i=1}^d (t-\theta_i)^{b_{i,j}} : j = 0,\dots, r$. 
\end{definition}

The following is immediate from \ref{it:basisrk} of Theorem~\ref{thm:frac_cosp}:
\begin{corollary}\label{cor:phiMfactor}
Let $K$ be $H$-cospectral in $M$. Then the characteristic polynomial $\phi(M) = \phi(M,t)$ can be decomposed as 
\[ \phi(M,t) = \prod_{j=0}^r P_j(t). \]
\end{corollary}

$H$-cospectrality becomes the most useful when $H$ has $|K|$ distinct eigenvalues. That is, when all $F_j$s have rank 1. Let us fix an eigenvector $v_j \in \im F_j$ of $H$ for each $j$.

\begin{lemma}\label{lem:K_distinct} Assume $H$ has $K$ distinct eigenvalues. Then
\begin{enumerate} 
\item
$P_j(t) = p_{v_j}(t)$ for all $1 \leq j \leq r = |K|$, where $p_{v_j}$ is the minimal polynomial of $M$ relative to $v_j$, that is, the smallest degree monic polynomial $p$ such that $p(M)\hat{v_j} = 0$, or, equivalently, the characteristic polynomial of the action of $M$ restricted to the space $\langle v_j, M v_j, M^2 v_j, \dots, \rangle$.  

\item $(\theta_a,\theta_b) \in \Phi_K$, if and only if there is a $j$ such that $\theta_a, \theta_b$ are both roots of $P_j(t)$. 
\end{enumerate}
\end{lemma}

\begin{proof} It is well-known that the relative minimal polynomial has only simple roots, and $p_{v_j}(\theta_i) = 0$ if and only if $E_i \hat{v_j} \neq 0$. Since $\dim\im F_j = 1$, the integers $b_{i,j}$ can only be 0 or 1, so $P_j$ has only simple roots. And $\theta_i$ is a root of $P_j(t)$ if and only if $E_i \hat{v_j} \neq 0$. Thus $p_{v_j}$ and $P_j$ have exactly the same roots, and they are both monic, so they are equal. 

To prove the 2nd part, note that $(\theta_a,\theta_b) \in \Phi_K$ if and only if $E_a D_K E_b \neq 0$, according to Definition~\ref{def:partition}. Here 
\[E_a D_K E_b w = E_a \widehat{\widetilde{D_K} \widetilde{E_b w}} = \sum_j E_a \widehat{F_j \widetilde{E_b w}}.\] Thus if $E_a D_K E_b w \neq 0$ then $E_a \widehat{F_j \widetilde{E_b w}} \neq 0$ for some $j$. Since $\im F_j$ is 1-dimensional, this implies both that $E_a \hat{v_j} \neq 0$ and that $F_j \widetilde{E_b w} \neq 0$. By Lemma~\ref{lem:imE} we see that there exists a vector $z \in \R^K$ such that $\widetilde{E_b w} = \widetilde{E_b  \hat{z}} = \tilde{E_b} z$, and so $0 \neq F_j \tilde{E_b} z= \tilde{E_b} F_j z$, which implies $\tilde{E_b} v_j \neq 0$ and hence $E_b \hat{v_j} \neq 0$. We have already established $E_a \hat{v_j} \neq 0$, and so we find that both $\theta_a$ and $\theta_b$ are roots of $P_j(t)$.

To see the converse direction, first note that
Claim~\ref{claim:nonzeroEv} implies $E_i \hat{v_j} \neq 0$  if and only if $\tilde{E_i}v_j = \widetilde{E_i \hat{v_j}}$ contains a component parallel to $v_j$. But since $\tilde{E_i} v_j = \tilde{E_i} F_j v_j = F_j \tilde{E_i}v_j \in \im F_j$ the latter of which is 1-dimensional, we find that $E_i \hat{v_j} \neq 0$ if and only if $v_j$ is an eigenvector of $\widetilde{E_i}$. If $\theta_a$ and $\theta_b$ are both roots of $P_j(t)$ then $E_a \hat{v_j}$ and $E_b \hat{v_j}$ are both non-zero. Hence $D_K E_b \hat{v_j} = \widehat{ \tilde{E_b} v_j}$ is a non-zero constant multiple of $\hat{v_j}$, and then $E_a D_K E_b \hat{v_j}$ is also non-zero, implying that $(\theta_a, \theta_b) \in \Phi_K$
\end{proof}

\begin{corollary}
If $K$ is $H$-cospectral in $M$ for some $H$ that has no multiple eigenvalues, then $\PK$ is the most refined partition where for each $j$ the roots of $P_
j(t)$ fall in the same part.
\end{corollary}

\subsection{Gluing and Diagonal perturbation}

\begin{theorem}\label{thm:gluing}
Let $X = X_1 \cup X_2$ a splitting of the index set $X$ such that $K = X_1 \cap X_2$. Let $M_1$ and $M_2$ be two matrices supported on $X_1$ and $X_2$ respectively. Suppose that $K$ is $H$-cospectral in both $M_1$ and $M_2$ for some $H \in \C^{K \times K}$. Then $K$ is also $H$-cospectral in $M = M_1 + M_2$.
\end{theorem}

\begin{proof} By Theorem~\ref{thm:frac_cosp} it suffices to show that $H$ commutes with $\widetilde{M^k}$ for all $k$. Note that $M^k$ can be expressed as
\[
M^k = \sum_{\sum(a_i+b_i)=k} M_1^{a_1}M_2^{b_2} M_1^{a_2} M_2^{b_2}\cdots
\]
Let $\Pi \in \R^{K \times X}$ denote the projection onto $K$, that is $\Pi$ has 1 on its diagonal in $S$. Then $\tilde{N} = \Pi N \Pi^{\tr}$ for any $X\times X$ matrix $N$. Further,  examining the matrix multiplication one can see that for any matrices $N_i$ supported on $X_i$ for  $i = 1,2$, on has 
\[
N_1 N_2 \Pi^{\tr} = N_1 \Pi^{\tr} \Pi N_2 \Pi^{\tr} \mbox{ and } N_2 N_1 \Pi^{\tr} = N_2 \Pi^{\tr} \Pi N_1 \Pi^{\tr} 
\]
Thus
\[
\widetilde{M^k} = \Pi M^k\Pi^{\tr} = \sum_{\sum(a_i+b_i)=k} \Pi M_1^{a_1}\Pi^{\tr} \Pi  M_2^{b_2}\Pi^{\tr} \Pi M_1^{a_2}\Pi^{\tr} \Pi  M_2^{b_2}\Pi^{\tr} \cdots
=  \sum_{\sum(a_i+b_i)=k} \widetilde{M_1^{a_1}}\widetilde{ M_2^{b_2}}\widetilde{M_1^{a_2}}\widetilde{M_2^{b_2}}\cdots.
\]
By assumption, $H$ commutes with both $\widetilde{M_1}$ and $\widetilde{M_2}$, so it also commutes with $\widetilde{M^k}$.
\end{proof}

In the next section we will study a diagonal perturbation of $M$ in the form of $M + Q D_K$. Using the previous theorem with $X_1 = X, X_2 = K, M_1 = M, M_2 = D_K$ and noting that $K$ is $H$-cospectral in $D_K$ for any $H$, we get the following.

\begin{corollary}\label{cor:potential}
If $K$ is $H$-cospectral in $M$ then it is also $H$-cospectral in $M + Q D_K$ for any $Q \in \R$.
\end{corollary}

\subsection{Diagonal perturbation}

Let $H$ be a normal matrix in $\C^{K \times K}$ that has $|K|$ distinct eigenvalues and an orthonormal eigenbasis $v_1, v_2, \dots, v_{|K|}$. In this section we assume that $K$ is $H$-cospectral in $M$. Let $\F \geq \Q$ denote the smallest number field containing all entries of  $M$ and all roots of $\phi(H)$. Then $v_1,\dots, v_{|K|} \in \F^{|K|}$. 

Let $Q \in \R$ denote a transcendental number that is algebraically independent of $\F$, and consider $M^K = M + Q \cdot D_K \in \R[Q]^{X \times X}$. Then, according to Corollary~\ref{cor:potential}, $K$ is also $H$-cospectral in $M^K$. In particular, according to Corollary~\ref{cor:phiMfactor} and Lemma~\ref{lem:K_distinct}, we have the factorization
\begin{equation}\label{eq:mk_factor} \phi(M^K) = P_0(t) \cdot \prod_{j=1}^{|K|} P_j(t) \end{equation} where $P_j$ is the minimal polynomial of $M^K$ relative to $\hat{v}_j$ for $j=1,\dots, |K|$. 

\begin{claim}\label{clm:deg_tr}
$P_0 \in \F[t]$ and $P_j \in \F[Q,t]$ for $j = 1,\dots,|K|$. Furthermore, the $Q$-degree of $P_j$ is 1  and $\Tr P_j - Q \in \F$ for each $j=1,\dots,|K|$
\end{claim}

\begin{proof}
Since $P_j$, for $j \geq 1$, is a relative minimal polynomial of a matrix in $\F(Q)^{X \times X}$ relative to a vector with entries in $\F$, it is automatic that $P_j \in \F(Q)[t]$. Then \eqref{eq:mk_factor} implies the same for $P_0$. However, $\F(Q)$ is the quotient field of the ring $\F[Q]$ which is a UFD. Hence by Gauss's Lemma the factorization \eqref{eq:mk_factor} is valid in $\F[Q,t]$.

The $Q$-degree of $\phi(M^K)$ is obviously $|K|$. It is easy to see that each $P_j : j \geq 1$ is at least linear in $Q$ and, in fact, satisfies $\Tr P_j - Q \in \F$. (This follows from a short analysis of the $Q^s$-terms in the linear combination of $(M^K)^s \hat{v}_j$ vectors that define the coefficients of $P_j$. See~\cite[Lemma 3.3]{eisenberg2019pretty} for a detailed proof.) This is only possible if the $Q$-degree of $P_0$ is 0, and the $Q$-degree of the other $P_j$s are exactly 1.
\end{proof}

\begin{theorem}\label{thm:irreducible}
$P_j$ is irreducible in $\F[Q,t]$ (and hence also in $\F(Q)[t]$) for all $j \geq 1$.
\end{theorem}
\begin{proof}

Now suppose for a contradiction that, say, $P_1$ is reducible in $\F[Q,t]$. Since its $Q$-degree is 1, this means $P_1(t) = R(t) \cdot \tilde{P}_1(t)$ for some non-constant $R \in \F[t]$ and $\tilde{P}_1 \in \F[t,Q]$. Let $\theta$ be a root of $R(t)$, and let $m$ denote the multiplicity of $\theta$ as a root of $P_0(t)$.
Then $\theta$ is a root of $\phi(M^K)$ with multiplicity at least $m+1$ for all values of $Q$. Since $D_K$ is a projection matrix, Lemma~\ref{lem:vanishing} implies that there are $m+1$ orthogonal eigenvectors of $M$ (corresponding to the $\theta$ eigenvalue) that vanish on $K$. But then the multiplicity of $\theta$ in $P_0$ should have been at least $m+1$ according to its definition. This is a contradiction, hence each $P_j : j=1,2,\dots,|K|$ is indeed irreducible.  
\end{proof}

\section{Examples}\label{sec:examples}

In this section we apply the above techniques to provide explicit families of graphs where diagonal perturbation can be shown to induce pretty good fractional revival. We use the graph and its adjacency matrix interchangeably. The blueprint is the following:
\begin{itemize}
\item Find a family of graphs and a subset $K$ of the nodes which are $H$-cospectral for some normal matrix $H$ with only simple eigenvalues.
\item Add a transcendental diagonal perturbation to the nodes in $K$.
\item Identify the factorization $\Phi(M^K,t) = \prod P_j(t)$ of the characteristic polynomial as the relative minimal polynomials of the eigenvectors of $H$. 
\item Using this, compute the degrees and traces of the $P_j$ polynomials.
\item The transcendentality of the diagonal perturbation guarantees that the $P_j$ are irreducible over a certain field, hence each part in the partition $\PK$ coincides with the roots of one of the $P_j$s. 
\item Show that for some choice of $i,j$ we have $\Tr P_i / \deg P_i \neq \Tr P_j /\deg P_j$. This implies that $\PK$ is non-degenerate, and hence there is pretty good state transfer relative to $K$. 
\end{itemize}

\subsection{PGFR on 2 nodes}

Here we present two infinite families of graphs, both of which are built on a path of length $2k+1$. In $S_{k,m}$ we take a path with $2k+1$ edges, and add a loop edge of weight $m$ at the $k+1$st node, that is, one of the nodes forming the middle edge. In $T_k$ we take a path with $2k+1$ edges and add two new nodes. One is going to be adjacent to nodes $k$ and $k+1$ on the path, the other is going to be adjacent to node $k+3$ on the path.

In both families, the endpoints of the path exhibit pretty good fractional revival under magnetic fields applied to them.
\vspace{1cm}
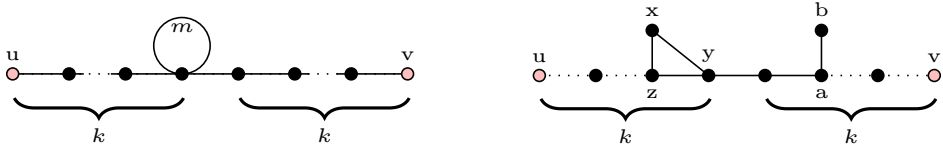
\begin{figure}[!h]
    
    \hspace*{1.5cm}
     \begin{tikzpicture}[transform canvas={scale=1.5}]
        
        \draw[dotted] (0,0) -- (3.5,0); 
       \draw (0,0) -- (0.63,0);
       \draw (0.87,0) -- (2.63,0);
       \draw (2.87,0) -- (3.5,0);
       \draw (1.5,0.25) circle (0.25) node[anchor=south]{\tiny $m$};

        \draw[fill=pink] (0,0) circle (1.5pt) node[anchor=south]{\tiny u};
        \draw[fill=black] (0.5,0) circle (1.5pt);
        \draw[fill=black] (1,0) circle (1.5pt);
        \draw[fill=black] (1.5,0) circle (1.5pt);
        \draw[fill=black] (2,0) circle (1.5pt);
        \draw[fill=black] (2.5,0) circle (1.5pt);        
        \draw[fill=black] (3,0) circle (1.5pt);
        \draw[fill=pink] (3.5,0) circle (1.5pt) node[anchor=south]{\tiny v};
        
        \draw [thick, black, decorate, decoration={brace,amplitude=5pt,mirror}, xshift=0.4pt, yshift=-0.4pt] (0,-0.2) -- (1.5,-0.2) node[black,midway,yshift=-0.33cm] {\tiny $k$};
        \draw [thick, black, decorate, decoration={brace,amplitude=5pt,mirror}, xshift=0.4pt, yshift=-0.4pt] (2,-0.2) -- (3.5,-0.2) node[black,midway,yshift=-0.33cm] {\tiny $k$}; 

        \end{tikzpicture}
        
        \vspace{-.4cm}
        \hspace*{8.5cm}
        \begin{tikzpicture}[transform canvas={scale=1.5}]

        \draw[dotted] (0,0) -- (1.0,0);
        \draw (1,0.4) -- (1.5,0) (1,0) -- (1,0.4);
        \draw (2.5,0) -- (2.5,0.4) (1,0) -- (2.5,0);
        \draw[dotted] (2.5,0) -- (3.5,0);
        
         \draw[fill=pink] (0,0) circle (1.5pt) node[anchor=south]{\tiny u};
        \draw[fill=black] (0.5,0) circle (1.5pt);
        \draw[fill=black] (1,0) circle (1.5pt) node[anchor=north]{\tiny z};
        \draw[fill=black] (1,0.4) circle (1.5pt) node[anchor=south]{\tiny x};
        \draw[fill=black] (1.5,0) circle (1.5pt) node[anchor=south]{\tiny y};
        \draw[fill=black] (2,0) circle (1.5pt);
        \draw[fill=black] (2.5,0) circle (1.5pt) node[anchor=north]{\tiny a};        
        \draw[fill=black] (2.5,.4) circle (1.5pt) node[anchor=south]{\tiny b};
        \draw[fill=black] (3,0) circle (1.5pt);
        \draw[fill=pink] (3.5,0) circle (1.5pt) node[anchor=south]{\tiny v};

       \draw [thick, black, decorate, decoration={brace,amplitude=5pt,mirror}, xshift=0.4pt, yshift=-0.2pt] (0,-0.2) -- (1.5,-0.2) node[black,midway,yshift=-0.33cm] {\tiny $k$};
        \draw [thick, black, decorate, decoration={brace,amplitude=5pt,mirror}, xshift=0.4pt, yshift=-0.2pt] (2,-0.2) -- (3.5,-0.2) node[black,midway,yshift=-0.33cm] {\tiny $k$}; 
        
        \end{tikzpicture}
    \vspace{1cm}    
    \caption{Two families of graphs with cospectral pairs: $S_{k,m}$ (left) and $T_k$ (right).}
    \label{fig:combver}
\end{figure}

\begin{theorem}\label{thm:main_K2}
Let $M$ denote the adjacency matrix of either $S_{k,m}$ or $T_k$, and $K = \{u,v\}$ be the endpoints of the path, as in the figure. Let $Q$ be a transcendental real number. Then $M^K = M + Q \cdot D_K$ exhibits PGFR with respect to $K$.
\end{theorem}

The proof of this result requires some preparation, as outlined in the blueprint above. The first step is to verify fractional cospectrality of $u$ and $v$.

\begin{lemma}\label{lem:uv_cospectral}
Let $M$ denote the adjacency matrix of either $S_{k,m}$ or $T_k$. Then for all $j \geq 0$
\[ (M^j)_{u,u} - (M^j)_{v,v} = c \cdot (M^j)_{u,v} \]
where $c = m$ in the case of $S_{k,m}$ and $c = 2$ in the case of $T_k$. 
\end{lemma}

\begin{proof}
Let $n$ be the size of the matrix $M$. Since $M^j$, for any $j \geq n$ is a linear combination of $M^0, M^1, \dots, M^{n-1}$, it suffices to check the equation for $j=0,\dots, n-1$. 

First we consider the $S_{k,m}$ case. Here $n = 2k+2$. For $j=0,1,\dots,2k$ it is clear that $(M^j)_{u,u} = (M^j)_{v,v}$ and $(M^j)_{u,v} = 0$. For $j = 2k+1$ we have $(M^{2k+1})_{u,u} - (M^{2k+1})_{v,v} = m$ and $(M^{2k+1})_{u,v} = 1$. Hence $c=m$ satisfies the equation for all $j=0,\dots, n-1$. 

Next we consider $T_k$. Here $n = 2k+4$. Clearly, $(M^j)$ counts walks of length $j$ in the graph. This graph is symmetric, except for the $xy$ edge. For $j < 2k+4$ no $v \to v$ walk of length $j$ can use the $xy$ edge, and for $j < 2k+1$ no $u \to u$ walk of length $j$ can use the $xy$ edge. Hence for $j < 2k+1$ we have $(M^j)_{u,u} = (M^j)_{v,v}$, and for $2k+1 \leq j \leq 2k+3$ we have $(M^j)_{u,u} - (M^j)_{v,v} = $ the number of such $u\to u$ walks of length $j$ that do use the $xy$ edge.  For $j=2k+1$ this number is 2 since you can only use this edge once and thus the rest of the walk has to be straight between $u$ and $x$ (respectively between $u$ and $y$). For $j=2k+2$ such a walk will have to use the $xy$ edge twice due to parity constraints. But then it has to use it twice back-to-back due to distance constraints. Hence there are only two such walks: $u \to x \to y \to x \to u$ and $u \to y \to x \to y \to u$. For both $j=2k+1$ and $j=2k+2$ there is clearly only a single $u\to v$ walk of this length. Hence $(M^j)_{u,u} - (M^j)_{v,v} = 2 \cdot (M^j)_{u,v}$ for $j=0,1,\dots, 2k+2$. 

It remains to check the same expression holds for $j=2k+3$. Due to parity and length constraints, each $u\to u$ walk of length $2k+3$ must ``go around'' the $xyz$ triangle exactly once. Reversal of the walk yields a bijection between those walks that go around clockwise and those that do so counterclockwise. By similar considerations, each $u \to v$ walk of length $2k+3$ must traverse the $zy$ edge in this direction. ****
\end{proof}

\begin{corollary}\label{cor:uv_example}~
\begin{enumerate}
\item According to Remark~\ref{rem:equiv_def}, in both families $K = \{u,v\}$ is $H$-cospectral for a 2-by-2 matrix $H$ whose eigenvectors are $(p,q)$ and $(-q,p)$ where $p/q - q/p = c$. 
\item Furthermore, according to Corollary~\ref{cor:phiMfactor} and Lemma~\ref{lem:K_distinct}, as long as $c \neq \pm \infty$, the characteristic polynomial factors as $\Phi(M,t) = P_0(t) P_1(t) P_2(t)$ where $P_1(t)$ (respectively $P_2(t)$) is the minimal polynomial of $M$ relative to $v_1$ (respectively $v_2)$  where these vectors are defined as 
\[ v_1(x) = \left \{ \begin{array}{cc} p/q :& x=u \\ 1 :& x=v \\ 0 :& x\neq u,v \end{array} \right.  \hspace{1cm} v_2(x) = \left \{ \begin{array}{cc} -q/p :& x=u \\ 1 :& x=v \\ 0 :& x\neq u,v \end{array} \right. \]
\item Then, according to Corollary~\ref{cor:potential}, $K$ is also $H$-cospectral in $M^K = M + Q D_K$, and $\Phi(M^K,t) = P^K_0(t) P^K_1(t,Q) P^K_2(t,Q)$ where $P^K_1$ and $P^K_2$ are degree 1 in $Q$, and for any choice of a transcendental $Q_0$, the one-variable polynomials $P^K_i(t,Q_0) : i = 1,2$ are irreducible over $\Q(p/q, Q_0)$.
\end{enumerate}
\end{corollary}

Next, we show that the traces of $P^K_1$ and $P^K_2$ are not equal to each other.

\begin{lemma}\label{lem:STfamily_trace}
With the above notation, $\deg P^K_1 = \deg P^K_2$ but $\Tr P^K_1 \neq \Tr P^K_2$. In particular they are not the same polynomial.
\end{lemma}

\begin{proof}
Since $c$ in Lemma~\ref{lem:uv_cospectral} is always an integer and $p/q, -q/p$ are the roots of $x-1/x= c$, it follows that $p/q$ and $-q/p$ are always quadratic integers that are each others' conjugates. This conjugation maps $v_1$ to $v_2$ in Corollary~\ref{cor:uv_example} and thus also $P^K_1$ to $P^K_2$. So these polynomials always have the same degree. They are both linear in $Q$, and, by definition, substituting $Q=0$ into them we get $P_1$ and $P_2$ respectively. So it is sufficient to show that $\Tr P_1 \neq \Tr P_2$. We will do so by examining specific coordinates of $M^j v_1$ and $M^j v_2$ for $j \leq k+2 = n/2$. 

We outline the computation for the $T_k$ case only, as the $S_{k,m}$ case follows in a similar but simpler way. For $T_k$ we saw that $c=2$ and thus $p/q = 1+\sqrt{2}$ and $-q/p = 1-\sqrt{2}$. The following are easily checked by direct calculation of walk counts in $T_k$:
\begin{align*}
(M^j)_{u,y} =& \left\{ \begin{array}{cl} 0 :& j < k \\ 1:& j=k \\ 1 :& j= k+1 \\ k+3 :^*& j =k+2 \end{array} \right. &
(M^j)_{v,y} =& \left\{ \begin{array}{cl} 0 :& j \leq k \\ 1 :& j= k+1 \\ 0 :& j =k+2 \end{array} \right. \\
(M^j)_{u,b} =& 0 : j \leq k+2 &
(M^j)_{v,b} =& \left\{ \begin{array}{cl} 0 :& j < k \\ 1:& j=k \\ 0 :& j= k+1 \\ k+1 :^*& j =k+2 \end{array} \right. 
\end{align*}
The entries marked by $(*)$ can be seen by noting that a walk of the appropriate length will use exactly one edge more than once. Specifying which edge this is determines the walk. In the $u \to y$ case this can be any of the edges on the $uy$ path, the $xz$ edge, or any of the other two edges incident to $y$. In the $v \to b$ case this can be any of the edges on the $va$ path, or one of the other two edges incident to $a$. 

Hence we get that 
\begin{align*}
M^j v_1|_{y,b} =& \left\{ \begin{array}{cl} (0,0):& j < k \\ (1 + \sqrt{2}, 1) :& j = k\\  (2+\sqrt{2},0)  :& j = k+1 \\ ( k+3 +(k+3)\sqrt{2}, k+1)  :& j = k+2 \end{array} \right. \\
M^j v_2|_{y,b} =& \left\{ \begin{array}{cl} (0,0):& j < k \\ (1 - \sqrt{2}, 1) :& j = k\\  (2-\sqrt{2},0)  :& j = k+1 \\ ( k+3 -(k+3)\sqrt{2}, k+1)  :& j = k+2 \end{array} \right. 
\end{align*}

On one hand, these calculations imply that the degree of $P_1$ and $P_2$ are both at least $k+2$, but since their product has degree $2k+4$, their degrees must, in fact, be equal to $k+2$. If $P_1(t) = x^{k+2} + c_1 x^{k+1} + c_2 x^k + \dots$ then $k+1 + c_2 = 0$ from the $b$ coordinate of $0 = P_1(M)v_1$, and $(k+3)(1+\sqrt{2}) + c_1(2+\sqrt{2}) + c_2(1+\sqrt{2}) =0$ from the $y$ coordinate of $0 = P_1(M)v_1$. Hence $c_2 = -k-1$ and $c_1 = (2+2\sqrt{2})/(2+\sqrt{2}) = 
\sqrt{2}$. So $\Tr P_1 = -\sqrt{2}$. Similarly $\Tr P_2 = \sqrt{2}$. 
\end{proof}

We are finally ready to prove the validity of our examples.

\begin{proof}[Proof of Theorem~\ref{thm:main_K2}]
According to Corollary~\ref{cor:uv_example} the polynomials $P^K_1$ and $P^K_2$ are irreducible and distinct from $P^K_0$. By Lemma~\ref{lem:STfamily_trace} they are distinct from each other, thus it follows that none of these polynomials share roots. Then by part 2) of Lemma~\ref{lem:K_distinct}, the eigenvalue support consists entirely of pairs of roots of $P^K_1$ and pairs of roots of $P^K_2$. Thus, the partition $\PK = (\Pi_0, \Pi_1, \Pi_2)$ is such that $\Pi_1$ consists exactly of all roots of $P^K_1$ and $\Pi_2$ consists exactly of all roots of $P^K_2$. Now, Lemma~\ref{lem:STfamily_trace} combined with Theorem~\ref{thm:tr/deg} implies that the partition $\PK$ is non-degenerate. Hence, our result follows by Theorem~\ref{thm:pgfr_pk}.
\end{proof}

\subsection{Cyclic Symmetry}

\newcommand{\ord}{r}

Consider a graph $G$ with an automorphism $T : V(G) \to V(G)$ of order $\ord$. We denote its adjacency matrix by $M$. Let $K \subset V(G)$ be an orbit of cardinality $\ord$. In this section we establish certain conditions under which $M^K = M + Q \cdot D_K$ exhibits PGFR relative to $K$. This turns out to be easier than for the previous examples.

Since $T$ has order $\ord$, it must act as a cyclic permutation on any orbit of cardinality $\ord$, and on $K$ in particular. Let $H$ denote the $K \times K$ permutation matrix describing the (cyclic) action of $T$ on $K$. It is obvious then that $H$ commutes with $\widetilde{M^k}$ for any $k \geq 0$, and thus $K$ is $H$-cospectral in $M$, and also in $M^K$ by Corollary~\ref{cor:potential}.  
 
 The eigenvalues of $H$ are simple, with eigenvectors 
 \begin{equation}\label{eq:vk}
 v_k = (1,\rho^k, \rho^{2k}, \dots, \rho^{(\ord-1)k})  : k = 1,\dots, \ord, 
 \end{equation} 
 where $\rho = e^{\frac{2\pi i}{\ord}}$. Thus $\Phi(M^K,t) = P_0(t) \prod_{k=1}^{\ord} P_k(t,Q)$ and according to Lemma~\ref{lem:K_distinct}, for each $k=1, \dots, p$ the polynomial $P_k(t)$ is the relative minimal polynomial of $v_k$. According to Theorem~\ref{thm:irreducible} the polynomials $P_1, \dots, P_{\ord}$ are irreducible, so they have disjoint sets of roots unless they are the same polynomial. Hence in the partition $\PK$ each part is the set of roots of one of the $P_k$ polynomials. 

\begin{lemma}\label{lem:main_cyclic}
Let $G,T,K$ be as above. If either $\deg P_i \neq \deg P_j$ or $\Tr P_i \neq \Tr P_j$ for some $i,j \geq 1$ then $M^K$ exhibits PGFR relative to $K$.
\end{lemma}

\begin{proof}
By Claim~\ref{clm:deg_tr} each of the $P_k : k =1,\dots, \ord$ polynomials have $\Tr P_k - Q \in \Q(\rho)$. Since $Q$ is transcendental, if $\deg P_i \neq \deg P_j$ then $\Tr P_i / \deg P_i \neq \Tr P_j / \deg P_j$. On the other hand, if $\deg P_i =\deg P_j$ but $\Tr P_i \neq \Tr P_j$ then again clearly $\Tr P_i / \deg P_i \neq \Tr P_j / \deg P_j$. Thus, we are done by Theorem~\ref{thm:tr/deg}.
\end{proof}

\subsubsection{Orbits of Unequal Size}

 \begin{theorem}\label{thm:main_psymm}
 Let $G$ and $T$ as above. Let $d$ denote the largest distance between $K$ and any other node of $G$. If $\ord(d+1) > |V(G)|$ then $M^K = M+Q\cdot D_K$ exhibits PGFR with respect to $K$ for any transcendental $Q \in \R$. 
  \end{theorem}
 
 \begin{remark}
 For this condition to hold, it is necessary that not all orbits of $T$ have size $\ord$. 
 \end{remark}
 
 \begin{proof}
 Since there is a node of $G$ that is $d$ distance away from $K$, and since $v_{\ord} = (1,1,\dots, 1)$, we see that for each $j = 0, 1, \dots, d$ the vectors $(M^K)^j v_p$ have strictly growing support, hence they can't be linearly dependent. Thus $\deg P_p \geq d+1$. Let $j_0$ be such that $P_j$ has the smallest degree among $P_1, \dots, P_{\ord-1}$.  We have $|V(G)|  = \deg P_0 + \deg P_p + \sum_{j=1}^{\ord-1} \deg P_j \geq d+1 + (p-1)\deg P_{j_0}$, hence $\deg P_{j_0} \leq (|V(G)|-d-1)/(p-1) < d+1$ according to the conditions on $d$ and $|V(G)|$. Thus $\deg P_{j_0} < d+1 \leq \deg P_p$ and thus the statement follows from Lemma~\ref{lem:main_cyclic}. 
 \end{proof}
 
\subsubsection{Cycles with Added Diamond Graphs}

In this section, we consider another family with general cyclic symmetry.  We define $G_{\ord}$ to be the graph obtain by starting with a cycle $C_{\ord}$ of order $\ord$, and attaching along each edge a diamond graph (two triangles sharing an edge).  The graph $G_5$ is pictured in Figure \ref{fig:cycle_diamond}.
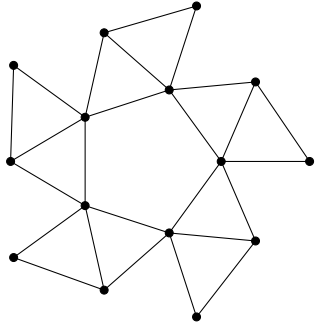
\begin{figure}[!h]
\begin{center}
\begin{tikzpicture}
\draw[fill=black] (0:1) circle (1.5pt) (72:1) circle (1.5pt) (144:1) circle (1.5pt) (216:1) circle (1.5pt) (288:1) circle (1.5pt) (36:1.8) circle (1.5pt) (108:1.8) circle (1.5pt) (180:1.8) circle (1.5pt) (252:1.8) circle (1.5pt) (324:1.8) circle (1.5pt) (0:2.175) circle (1.5pt) (72:2.175) circle (1.5pt)  (144:2.175) circle (1.5pt)  (216:2.175) circle (1.5pt) (288:2.175) circle (1.5pt) ;
\draw (0:1)--(72:1)--(144:1)--(216:1)--(288:1)--(360:1)--(36:1.8)--(72:1)--(108:1.8)--(144:1)--(180:1.8)--(216:1)--(252:1.8)--(288:1)--(324:1.8)--(360:1)--(0:2.175)--(36:1.8) (72:1)--(72:2.175)--(108:1.8) (144:1)--(144:2.175)--(180:1.8) (216:1)--(216:2.175)--(252:1.8) (288:1)--(288:2.175)--(324:1.8);
\end{tikzpicture}
\end{center}
\caption{The graph $G_5$.}\label{fig:cycle_diamond}
\end{figure}
Note that $G_\ord$ has the cyclic group of order $\ord$ as its automorphism group.  There are three orbits of the automorphism group, consisting of the vertices of degree 5, the vertices of degree 3, and the vertices of degree 2 respectively.  

\begin{theorem}
Let $K$ be the vertices of any single orbit, for instance the vertices of the central cycle $C_\ord$, and as above $M^K=M+Q\cdot D_K$ for transcendental $Q$.  Then $M^K$ exhibits PGFR with respect to $K$.
\end{theorem}

\begin{proof}
By Lemma~\ref{lem:main_cyclic} it suffices to show that $\Tr P_{\ord} \neq \Tr P_k$ for some other $1 \leq k \leq \ord -1$, where $P_k$ is the minimal polynomial of $M^K$ relative to the $v_k$ vector defined in \eqref{eq:vk}.

Let us define $A_{C_\ord}$ to be the adjacency matrix for $C_\ord$ and \[R=\begin{bmatrix}1&1&0&0&\cdots& 0\\0&1&1&0&\cdots&0\\0&0&1&1&\cdots&0\\\vdots&\vdots&\vdots&\ddots&\ddots&\vdots\\0&0&0&\cdots&1&1\\1&0&0&\cdots&0&1\end{bmatrix}\] and note that 
\[M^K = \begin{bmatrix}A_{C_\ord}+Q\cdot I&R&I\\R^T&0&I\\I&I&0\end{bmatrix}.\]
Let $\lambda_k=\rho^{-k}+\rho^k$ and observe that simple calculation yields 
\begin{align*}
A_{C_n}v_k &= \lambda_k v_k\\
R v_k &= (1+\rho^k) v_k\\
R^T v_k &= (1+\rho^{-k}) v_k.
\end{align*}
Thus we can verify that 
\[
\begin{bmatrix}A_{C_\ord}+Q\cdot I&R&I\\R^T&0&I\\I&I&0\end{bmatrix}\begin{bmatrix}a v_k\\b v_k\\ c v_k\end{bmatrix}=\begin{bmatrix}a' v_k\\b ' v_k \\ c' v_k\end{bmatrix}
\]
as an equation of $3\ord \times 3$ matrices, where
\[
\begin{bmatrix}\lambda_k+Q&1+\rho^k&1\\1+\rho^{-k}&0&1\\1&1&0\end{bmatrix} \begin{bmatrix} a \\ b \\ c \end{bmatrix}= \begin{bmatrix}a' \\ b' \\ c' \end{bmatrix}.
\]

It is easy to see that $\{ (a v_k, b v_k, c v_k) : a,b,c \in \R\}$ is the subspace generated by $(M^K)^j v_k : j = 0,1, \dots$, hence the roots of the relative minimal polynomial $P_k$ are exactly the eigenvalues of 
\[
N_k:=\begin{bmatrix}\lambda_k+Q&1+\rho^k&1\\1+\rho^{-k}&0&1\\1&1&0\end{bmatrix}.
\]
 Thus $P_k$ has degree 3, and we can see that $\Tr P_k = \Tr N_k = \lambda_k+Q$ for each $k$. Note in particular that $\Tr P_\ord \neq \Tr P_k$ for $k\neq \ord$, and hence Lemma~\ref{lem:main_cyclic} finishes the proof. 
\end{proof}
 
\section{Appendix}

\begin{lemma}\label{lem:imE}
Let $E \in \R^{X\times X}$ be a projection, and $K\subset X$. If $v = E w$ then there is a $w' \in \R^K$ such that $\tilde{v} = \tilde{E} w'$. Here we used the notation from Section~\ref{sec:h-cospectral}.
\end{lemma}

\begin{proof}
Let $F = E_{K \times X}$. Then $\{ \tilde{v} : v \in \im E\} = \im F$. So it is sufficient to show that $\im \tilde{E} = \im F$. In fact, $\im \tilde{E} \leq \im F$ by definition, so showing $\im F \leq \im \tilde{E}$ suffices.

Since $E$ is a projection, we can write $E = \sum_j v_j v_j^{\tr}$, and $\tilde{E} = \sum_j \tilde{v_j} \tilde{v_j}^{\tr}$. Thus \begin{equation}\label{eq:kerE}
\ker \tilde{E} = \{ w \in \R^K | \forall j: \tilde{v_j}^{\tr} w = 0\}.\end{equation} Since $\tilde{E}$ is self-adjoint, $\im \tilde{E}$ is the orthogonal complement of $\ker \tilde{E}$. From \eqref{eq:kerE} it is clear that $\im \tilde{E} = \sprod{\tilde{v_1}, \tilde{v_2}, \dots}$. On the other hand, $F = \sum \tilde{v_j} v_j^{\tr}$, so clearly $\im F \leq  \sprod{\tilde{v_1}, \tilde{v_2}, \dots} = \im \tilde{E}$.
\end{proof}

\begin{lemma}\label{lem:vanishing}
Let $M$ be a symmetric matrix and $N$ be a projection matrix. Suppose $\theta$ has multiplicity $k$ as an eigenvalue of $M+ Q \cdot N$ for all $Q \in \R$. Then $\ker N$ contains $k$ orthonormal eigenvectors of $M$, each with eigenvalue $\theta$.
\end{lemma}

\begin{proof}
Suppose we already exhibited $0 \leq j \leq k-1$ such orthonormal vectors, $w_1, \dots, w_j \in \ker N$. We exhibit one more as follows. Let $v_q$ be a unit length eigenvector of $M+Q \cdot N$ with eigenvalue $\theta$ that is orthogonal to $w_1, \dots, w_j$. Such vectors exist because the multiplicity of $\theta$ is more than $j$, and $w_1, \dots, w_j$ are all eigenvectors of $M+ Q \cdot N$ as well. Let $w$ be a subsequential limit of the $v_q$s as $Q \to 0$. Such a limit exists because of compactness. From now on we restrict $Q$ to such a subsequence. Clearly $w$ is unit length and orthogonal to $w_1, \dots, w_j$. Taking limit as $Q \to 0$ in $(M+Q \cdot N)v_q = \theta v_q$ gives that $M w= \theta w$.  It remains to show that $w \in \ker N$.

Taking the scalar product of both sides with $w$, and using the symmetry of $M$ as well as $M w = \theta w$, we obtain
\[ \theta \sprod{v_q, w}=  \sprod{M v_q, w} + Q \sprod{N v_q, w} = \sprod{v_q, M w} + Q \sprod{N v_q, w} = \sprod{v_q, \theta w} + Q \sprod{N v_q, w} . \]
Hence $\sprod{N v_q, w} =0$, from which we get $\sprod{N w,w}=\sprod{Nw, Nw} = 0$ after passing to the limit and using that $N$ is idempotent. Thus $N w =0$ as claimed.
\end{proof}



\end{document}